\documentclass[a4paper]{amsart}
\usepackage{graphicx}
\usepackage{amssymb}
\usepackage{amsmath}
\usepackage{amsthm}
\usepackage{amscd}
\usepackage[all,2cell]{xy}

\UseAllTwocells \SilentMatrices
\newtheorem{thm}{Theorem}[section]

\newtheorem{cor}[thm]{Corollary}
\newtheorem{lem}[thm]{Lemma}

\theoremstyle{definition}

\theoremstyle{remark}
\newtheorem{rem}[thm]{\bf Remark}
\numberwithin{equation}{section}

\begin{document}
\title[Graded self-injective algebras ``are" trivial extensions]
{Graded self-injective algebras ``are" trivial extensions}
\author[  Xiao-Wu Chen
] {Xiao-Wu Chen}
\thanks{This project was supported by Alexander von Humboldt Stiftung, and was
also partially supported by China Postdoctoral Science Foundation
(No.s 20070420125 and 200801230). The author also gratefully
acknowledges the support of K. C. Wong Education Foundation, Hong
Kong}
\thanks{E-mail:
xwchen$\symbol{64}$mail.ustc.edu.cn}
\maketitle
\date{}%
\dedicatory{}%
\commby{}%
\begin{center}
\end{center}

\begin{abstract}
For a positively graded artin algebra $A=\oplus_{n\geq 0}A_n$ we
 introduce its Beilinson algebra $\mathrm{b}(A)$. We prove that
if $A$ is well-graded self-injective, then the category of graded
$A$-modules is equivalent to the category of graded modules over the
trivial extension algebra $T(\mathrm{b}(A))$. Consequently, there is
a full exact embedding from the bounded derived category of
$\mathrm{b}(A)$ into the stable category of graded modules over $A$;
it is an equivalence if and only if the $0$-th component algebra
$A_0$ has finite global dimension.
\end{abstract}

\section{Introduction}
Let $R$ be a commutative artinian ring and let $A=\oplus_{n\geq
0}A_n$ be a positively graded artin $R$-algebra.  Set
$c=\mbox{max}\{n\geq 0\; |\; A_n\neq 0\}$. Throughout we will assume
that $A$ is nontrivially graded, that is, $c\geq 1$. We define the
\emph{Beilinson algebra $\mathrm{b}(A)$ of the graded algebra $A$}
to be the following upper triangular matrix algebra
  $$\mathrm{b}(A)=\begin{pmatrix} A_0 & A_1 & \cdots & A_{c-2} & A_{c-1}\\
                     0     &  A_0  & \cdots   & A_{c-3} & A_{c-2}\\
                       \vdots &  \vdots &   \ddots & \vdots &  \vdots \\
                       0 & 0  & \cdots  & A_0 & A_1\\
                         0 & 0 & \cdots  & 0 & A_0 \end{pmatrix}. $$
Here the multiplication of $\mathrm{b}(A)$ is induced from the one
of $A$.  This concept  originates from the following example: let
 $A$ be the exterior algebra over a field with the usual grading, then its
Beilinson algebra $\mathrm{b}(A)$ appeared in  Beilinson's study on
the bounded derived category of projective spaces (see \cite{Bei},
also see the algebras in \cite[Example 4.1.2]{Ba}, \cite[p.90]{H1}
and  \cite[Corollary 2.8]{O2}); note that the algebra
$\mathrm{b}(A)$ is different from the \emph{Beilinson algebra} in
\cite{KK} (also see \cite{Bei2} and \cite[p.332, Remark]{Ba}), while
they are derived equivalent.

\par \vskip 5pt

Denote by $A\mbox{-gr}$ the category of finitely generated  graded
left $A$-modules with morphisms preserving degrees. It is well known
that the algebra $A$ is self-injective (as an ungraded algebra) if
and only if every projective object in $A\mbox{-gr}$ is injective
(by \cite[Theorem 2.8.7]{NV} or \cite{GG}). In this case, we say
that the graded artin algebra $A$ is  \emph{graded self-injective}
(compare \cite{SZ}).

 \par \vskip 5pt

Let $T(\mathrm{b}(A))=\mathrm{b}(A)\oplus D(\mathrm{b}(A))$ be the
\emph{trivial extension algebra} of $\mathrm{b}(A)$, where $D$ is
the Matlis duality on finitely generated $R$-modules (\cite[Chapter
II \S 3]{ARS}). Note that $T(\mathrm{b}(A))$ is a graded algebra
such that ${\rm deg}\;\mathrm{b}(A)=0$ and ${\rm
deg}\;\mathrm{b}(A)=1$, and that $T(\mathrm{b}(A))$ is graded
self-injective (\cite[p.128, Proposition 3.9]{ARS} and \cite[p.62,
Lemma 2.2]{H1}).

 \par \vskip 5pt

We say that the graded artin algebra $A$ is \emph{left well-graded}
if for each nonzero idempotent $e\in A_0$, $eA_c\neq 0$. Dually one
has the notion of \emph{right well-graded algebras}. We say that the
graded algebra $A$ is \emph{well-graded} provided that it is both
left and right well-graded. For example if the $0$-th component
algebra $A_0$ is  local, then $A$ is well-graded. Note that for a
graded self-injective algebra $A$ it is left well-graded if and only
if it is right well-graded, thus well-graded, see Lemma 2.2. Clearly
the trivial extension algebra $T(\mathrm{b}(A))$ is well-graded,
since $D(\mathrm{b}(A))$ is a faithful (left and right)
$\mathrm{b}(A)$-module.

 \par \vskip 5pt

The following result is inspired by \cite[Chapter II, Example
5.1]{H1}, and it  somehow justifies the title.

\begin{thm}Let $A=\oplus_{n \geq 0}A_n$ be a well-graded self-injective
algebra. Then we have an equivalence of categories $A\mbox{-}{\rm
gr}\simeq T(\mathrm{b}(A))\mbox{-} {\rm gr}$.
\end{thm}

Note that the equivalence above may not be the \emph{graded
equivalence of algebras} in the sense of \cite{GG} (compare
\cite{Bo, Z}), that is, in general it does not commute with the
degree-shift automorphisms.

\par \vskip 5pt

Denote by $A\mbox{-}\underline{\rm gr}$ the stable category with
respect to projective modules. It has a natural triangulated
structure (\cite[Chapter I, section 2]{H1}). Denote by
$\mathrm{b}(A)\mbox{-mod}$ the category of finitely generated left
$\mathrm{b}(A)$-modules, $D^b(\mathrm{b}(A)\mbox{-mod})$ its bounded
derived category.  The following  generalizes a result by Orlov
\cite[Corollary 2.8]{O2}, which might be traced back to \cite{Bei,
BGG, Bei2} (consult \cite{Ba}).

\begin{cor}
Let $A=\oplus_{n\geq 0}A_n$ be a well-graded self-injective algebra.
Then we have a full exact embedding of triangulated categories
$A\mbox{-}\underline{\rm gr}\hookrightarrow
D^b(\mathrm{b}(A)\mbox{-}{\rm mod}).$ Moreover, it is an equivalence
if and only if the $0$-th component algebra $A_0$ has finite global
dimension.
\end{cor}

\begin{proof}
Note that by \cite[p.78, Proposition 2.7]{ARS}, the algebra  $A_0$
has finite global dimension if and only if  so does the Beilinson
algebra $\mathrm{b}(A)$.  Theorem 1.1 implies the natural
equivalence $A\mbox{-}\underline{\rm gr}\simeq
T(\mathrm{b}(A))\mbox{-} \underline{\rm gr}$ of triangulated
categories (by \cite[Chapter I, 2.8]{H1}). Thus the corollary
follows immediately from a result by Happel (\cite[Theorem
2.5]{H2}).
\end{proof}

\section{The Proof of Theorem 1.1}
Let $R$ be a commutative artinian ring, and let $D={\rm Hom}_R(-,
E)$ be the Matlis duality with $E$ the minimal injective
$R$-cogenerator (\cite[p.37-39]{ARS}).  Let $B$ be an artin
$R$-algebra and let $_BX_B$ be a $B$-bimodule such that $R$ acts on
$X$ centrally  and  $X$ is finitely generated both as a  left and
right $B$-module. The \emph{trivial extension $B\ltimes X$ of $B$ by
the bimodule $X$ } is defined as follows: as an $R$-module $B\ltimes
X=B\oplus X$, and the multiplication is given by $(b,m)(b',m')=(bb',
bm'+mb')$ (\cite[p.78]{ARS}). Then $B\ltimes X$ is a positively
graded $R$-algebra such that ${\rm deg} \; B=0$ and ${\rm deg}\;
X=1$. We will denote by $B\ltimes X\mbox{-gr}$ the category of
finitely generated graded left $B\ltimes X$-modules.

\vskip 5pt

 Consider \emph{the regular $B$-bimodule} $_BB_B$ and its dual
$B$-bimodule $D(B)=D(_BB_B)$, and thus the $B$-bimodule structure on
$D(B)$ is given such that for each $b\in B$ and $f\in D(B)={\rm
Hom}_R(B,E)$, $(bf)(x)=f(xb)$ and $(fb)(x)=f(bx)$ for all $x\in B$.
The trivial extension $T(B)=B\ltimes D(B)$ is simply referred as the
\emph{trivial extension algebra} of $B$. It is a symmetric algebra,
thus self-injective (\cite[p.128, Proposition 3.9]{ARS}). More
generally, given an automorphism $\sigma: B\longrightarrow B$ of
$R$-algebras, consider the \emph{twisted $B$-bimodule}
$_BB_B^\sigma$ such that the left $B$-module structure is given by
the multiplication as usual and the right $B$-module structure is
given by $xb:=x\sigma(b)$, for all $b\in B$ and $x\in
{_BB_B^\sigma}$. Note that since $\sigma$ is an $R$-algebra
automorphism, $R$ acts on the $B$-bimodule $_BB_B^\sigma$ centrally.
Denote by $D(B^\sigma)=D(_BB_B^\sigma)$ the dual $B$-bimodule and
the corresponding trivial extension $T(B^\sigma)=B\ltimes
D(B^\sigma)$ is called the \emph{twisted trivial extension algebra
of $B$ with respect to $\sigma$}. Note that $T(B^\sigma)$ is
self-injective, in general not symmetric (see Example (4) in
\cite{Far}).

\vskip 5pt

We observe the following result.

\begin{lem}
Use the notation above. We have an isomorphism of categories
$T(B)\mbox{-}{\rm gr} \simeq T(B^\sigma)\mbox{-}{\rm gr}$.
\end{lem}

\begin{proof}
Note that as $R$-modules $T(B^\sigma)=B\oplus D(B)$, and its
multiplication is given by $(b,f)\star (b',f')=(bb',
\sigma(b)f'+fb')$. Given a graded $T(B)$-module $M=\oplus_{n\in
\mathbb{Z}}M_n$, we endow a $T(B^\sigma)$-action on it as follows:
given a homogeneous element $m\in M$, define
$$(b,f)\star
m=\sigma^{|m|}(b)m+(f\circ \sigma^{-|m|}) m,$$ where $|m|$ denotes
the degree of $m$, and $f\circ \sigma^{-|m|}:
B\stackrel{\sigma^{-|m|}}\longrightarrow
B\stackrel{f}\longrightarrow E \in D(B)$ means the composite. It is
direct to check that ``$\star$" gives $M$ a graded
$T(B^\sigma)$-module structure. Furthermore this gives an
isomorphism (more than an equivalence) of categories
$T(B)\mbox{-}{\rm gr} \simeq T(B^\sigma)\mbox{-}{\rm gr}$.
\end{proof}

\vskip 5pt

 Let $A=\oplus_{n\geq 0}A_n$ be a positively graded artin algebra and
let $c=\mbox{max}\{n\geq 0\; |\; A_n\neq 0\}$. As in the
introduction we always assume that $c\geq 1$. Consider the category
$A\mbox{-gr}$ of finitely generated graded left $A$-modules. For a
graded $A$-module $M=\oplus_{n\in \mathbb{Z}}M_n$, its \emph{width}
$w(M)$ is defined to be ${\rm max}\{n\;| M_n\neq 0\} - {\rm
min}\{n\;| M_n\neq 0\}+1$ (for $M=0$, set $w(M)=0$). For example
$w(A)=c+1$, here we regard $A$ as a graded $A$-module via the
multiplication such that the identity $1_A$ is at the $0$-th
component. For a graded $A$-module $M=\oplus_{n\in \mathbb{Z}}M_n$,
denote by $M(1)$ its \emph{shifted module} which is the same as $M$
as ungraded modules, and which is graded such that $M(1)_n=M_{n+1}$.
This gives rise to the \emph{degree-shift automorphism} $(1):
A\mbox{-gr}\longrightarrow A\mbox{-gr}$. Denote by $(d)$ the $d$-th
power of $(1)$ for each $d\in \mathbb{Z}$ (\cite{NV}). Recall that
each indecomposable projective object in $A\mbox{-gr}$ is of the
form $Ae(d)$, where $e\in A_0$ is a primitive idempotent and $d\in
\mathbb{Z}$; dually each indecomposable injective object is of the
form $D(eA)(d)$, where $D(eA)$ is graded such that
$D(eA)_n=D(eA_{-n})$. For details, see \cite[section 5]{GG}.

\par \vskip 5pt

\begin{lem}
Let $A=\oplus_{n\geq 0}A_n$ be a graded self-injective algebra.
Assume that it is left well-graded. Then it is right well-graded.
\end{lem}

\begin{proof}
Note that $A$ is left well-graded if and only if $w(Ae)=c+1$ for
each primitive idempotent $e\in A_0$, thus if and only if $w(P)=c+1$
for each indecomposable projective object $P$ in $A\mbox{-gr}$.
Since $A$ is graded self-injective, the indecomposable injective
graded module $D(eA)$ is projective, and thus by above
$w(D(eA))=c+1$. Note that $w(D(eA))=w(eA)$, where $eA$ is considered
as a graded right $A$-module. Hence for each primitive idempotent
$e\in A_0$ we have $w(eA)=c+1$, and this shows that $A$ is right
well-graded.
\end{proof}

We will divide the proof of Theorem 1.1 into several easy results.
Let $A=\oplus_{n\geq 0}A_n$ be a graded artin algebra and let
$\mathrm{b}(A)$ be its Beilinson algebra. Consider the following
$R$-module
  $$\mathrm{x}(A)=\begin{pmatrix} A_c & 0 & \cdots & 0 & 0\\
                     A_{c-1}     &  A_c  & \cdots   & 0 & 0\\
                       \vdots &  \vdots &   \ddots & \vdots &  \vdots \\
                       A_2 & A_3  & \cdots  & A_c & 0\\
                         A_1 & A_2 & \cdots  & A_{c-1} & A_c \end{pmatrix}. $$
Note that there is a natural $\mathrm{b}(A)$-bimodule structure on
$\mathrm{x}(A)$, induced from matrix multiplication and the
multiplication of $A$; moreover, $R$ acts on $\mathrm{x}(A)$
centrally. Consider the trivial extension
$\mathrm{t}(A)=\mathrm{b}(A)\oplus \mathrm{x}(A)$, which is a graded
algebra as above.

\begin{lem}
There is an equivalence of categories $A\mbox{-}{\rm gr}\simeq
\mathrm{t}(A)\mbox{-}{\rm gr}$. Moreover, $A$ is left well-graded if
and only if $\mathrm{t}(A)$ is.
\end{lem}

\begin{proof}
Define a functor $\Phi:A\mbox{-}{\rm gr}\longrightarrow
\mathrm{t}(A)\mbox{-}{\rm gr} $ as follows: for $M=\oplus_{n\in
\mathbb{Z}}M_n\in A\mbox{-gr}$, set $\Phi(M)=\oplus_{n\in
\mathbb{Z}}\Phi(M)_n$ with
$\Phi(M)_n=\bigoplus_{i=nc}^{(n+1)c-1}M_i$, and there is a natural
graded $\mathrm{t}(A)$-module structure on $\Phi(M)$ (using the
multiplication rule of matrices on column vectors; here the elements
in $\Phi(M)_n$ are viewed as column vectors of size $c$, and note
that $c\geq 1$); the action of $\Phi$ on morphisms is the identity.
To construct the inverse,  for each $0\leq r\leq c-1$, set
$e_{rr}\in \mathrm{b}(A)$ to be the elementary matrix having the
$(r+1, r+1)$ entry $1$ and elsewhere $0$. Define a functor $\Psi:
\mathrm{t}(A)\mbox{-}{\rm gr} \longrightarrow A\mbox{-}{\rm gr}$
sending a graded $\mathrm{t}(A)$-module $N=\oplus_{n\in
\mathbb{Z}}N_n$ to $\Psi(N)=\oplus_{n\in \mathbb{Z}}\Psi(N)_n$ such
that $\Psi(N)_{ic+r}=e_{rr}N_{i}$ for $i\in \mathbb{Z}$ and $0\leq
r\leq c-1$; on $\Psi(N)$ there is a natural graded $A$-module
structure. Then it is direct to check that $\Phi$ and $\Psi$ are
mutually inverse to each other. Note that one may have a more
conceptual proof of the equivalence above by \cite[Theorem
2.12]{Sie} (compare \cite[Example 3.10]{Z} and \cite{Bo}). \vskip
3pt

 For the second statement, take $1_{A_0}=\sum_{i=1}^l e_i$ to be  a
 decomposition of unity into primitive idempotents, and thus every primitive idempotent
 of $A_0$ is conjugate to one of $e_i$'s. Hence $A$ is left well-graded
 if and only if $e_iA_c\neq 0$ for each $1\leq i\leq l$. However $1_{\mathrm{b}(A)}=\sum_{r=0}^{c-1}\sum_{i=1}^l
 e_{rr}e_i$ is a decomposition of unity in $\mathrm{b}(A)$ into primitive idempotents,
 and hence $\mathrm{t}(A)$ is left well-graded if and only if $e_{rr}e_i\mathrm{x}(A)\neq 0$
 for each $0\leq r\leq c-1$ and $1\leq i\leq l$.
 Note that $e_{rr}e_i\mathrm{x}(A)=e_iA_c$ and then we are done.
\end{proof}

Consider the trivial extension  $T=B\ltimes  X$ of an artin
$R$-algebra $B$ by a (nonzero) $B$-bimodule $X$ as above. Take $e\in
B$ to be an idempotent such that $eBe$ is the \emph{basic algebra}
associated to $B$ (\cite[p.35]{ARS}). Thus $eXe$ has the induced
$eBe$-bimodule structure and we have an identification of (graded)
algebras $eTe=eBe\ltimes eXe$. The following result is immediate
from the Morita equivalence between the algebras $B$ and $eBe$.

\begin{lem}
Use the notation above. We have an equivalence of categories
$T\mbox{-}{\rm gr} \simeq eTe\mbox{-}{\rm gr}$. Moreover $T$ is left
well-graded if and only if so is $eTe$.
\end{lem}

The key observation is as follows.

\begin{lem}
Let $T=B\ltimes X$ be a trivial extension as above. Assume that $B$
is a basic algebra and $T$ is well-graded self-injective. Then there
is an isomorphism of $B$-bimodules $X\simeq D(B^\sigma)$ for some
$R$-automorphism $\sigma$ on $B$. In particular, there is an
isomorphism
 $T\simeq T(B^\sigma)$ of graded algebras.
\end{lem}

\begin{proof}
Take $1_B=\sum_{i=1}^le_i$ to be a decomposition of unity into
primitive idempotents. Since $B$ is basic, the set $\{Te_i(d)\; |\;
1\leq i\leq l, d\in \mathbb{Z}\}$ forms a complete set of pairwise
non-isomorphic projective objects in $T\mbox{-gr}$. Dually,
$\{D(e_iT)(d)\; |\; 1\leq i\leq l, d\in \mathbb{Z}\}$ forms a
complete set of pairwise non-isomorphic injective objects in
$T\mbox{-gr}$. Since $T$ is well-graded, all these modules have
width $2$. Since $T$ is graded self-injective, we have an
isomorphism of graded $T$-modules $Te_i\simeq D(e_{s(i)}T)(-1)$,
where $s:\{1, \cdots ,l\}\longrightarrow \{1, \cdots ,l\}$ forms a
permutation. In particular, we have isomorphisms $Xe_i\simeq
D(e_{s(i)}B)$ of left $B$-modules for each $1\leq i\leq l$. Since
$B$ is basic,  we deduce an isomorphism of left $B$-modules
$_BX\simeq D(B_B)$. Similarly we have an isomorphism $X_B\simeq
D(_BB)$ of right $B$-modules.

Consider the dual $B$-bimodule $M=D(_BX_B)$. We have isomorphisms
$_BM \simeq {_BB}$ and $M_B\simeq B_B$. It is a good exercise to
deduce from these isomorphisms that there is an isomorphism of
$B$-bimodules $M\simeq {_BB_B^\sigma}$ for some $R$-automorphism
$\sigma$ on $B$, and thus $X \simeq D(M)\simeq D(B^\sigma)$. We are
done.
\end{proof}

\begin{rem}
The same argument as in the proof above yields the following result
immediately, which we will not use, and which seems of independent
interest. Let $A=\oplus_{n\geq 0} A_n$ be a graded artin algebra
with $A_0$ basic. Set $c={\rm max}\{n\geq 0\; |\; A_n\neq 0\}$.
 Then the following statements are equivalent: \\
(1).\quad  $A$ is \emph{graded Frobenius}, that is, $_AA\simeq
D(A_A)(-c)$ as graded  left $A$-modules. \\
(2). \quad $A$ is graded
self-injective and $A_c$ is a faithful left $A_0$-module.\\
(3). \quad $A$ is well-graded self-injective.
\end{rem}

\vskip 5pt

 \noindent {\bf Proof of Theorem 1.1.}\quad By Lemma 2.3 we have an
 equivalence $A\mbox{-gr}\simeq \mathrm{t}(A)\mbox{-gr}$, and thus the algebra
 $\mathrm{t}(A)$ is well-graded self-injective. Set
 $B=e\mathrm{b}(A)e$ to be the basic algebra associated to the
 Beilinson algebra $\mathrm{b}(A)$. Thus by Lemma 2.4 $\mathrm{t}(A)\mbox{-gr}
 \simeq (B\ltimes X)\mbox{-gr}$  for some (nonzero) $B$-bimodule $X$,
 moreover, the trivial extension $T=B\ltimes X$ is well-graded
 self-injective. By Lemma 2.5, we have an isomorphism of graded
 algebras $T\simeq T(B^\sigma)$, and thus combining it with Lemma
 2.1 we deduce that $T\mbox{-gr}\simeq T(B)\mbox{-gr}$. Now applying
 Lemma 2.4 again we have $T(B)\mbox{-gr}\simeq
 T(\mathrm{b}(A))\mbox{-gr}$ (note that $B=e\mathrm{b}(A)e$ and we have  a natural $B$-bimodule isomorphism
 $D(B)\simeq eD(\mathrm{b}(A))e$), and thus we get the desired equivalence
 $A\mbox{-gr}\simeq T(\mathrm{b}(A))\mbox{-gr}$.  \hfill $\square$

\vskip 10pt

 \noindent {\bf Acknowledgement}\quad The author  would
like to thank Dr. Jiwei He very much for pointing out Remark 2.6 and
the reference \cite{Far} to him.

\bibliography{}

\vskip 10pt

 {\footnotesize \noindent Xiao-Wu Chen, Department of
Mathematics, University of Science and Technology of
China, Hefei 230026, P. R. China \\
Homepage: http://mail.ustc.edu.cn/$^\sim$xwchen \\
\emph{Current
address}: Institut fuer Mathematik, Universitaet Paderborn, 33095,
Paderborn, Deutschland}

\end{document}